\documentclass[11pt]{article}
\usepackage{calc,amssymb,amsmath,ulem,stmaryrd,fullpage, amsthm}
\usepackage{enumerate}
\usepackage{alltt}
\RequirePackage[dvipsnames,usenames]{color}

\input{xy}
\xyoption{all}
\usepackage{graphicx}
\usepackage[all,cmtip]{xy}
\usepackage{verbatim}
\usepackage[left=0.95in,top=0.95in,right=0.95in,bottom=0.95in]{geometry}
\usepackage{tikz-cd}
\usepackage{hyperref,url}
\usepackage{biblatex}
\hypersetup{colorlinks=true,
urlcolor=blue,
linkcolor=cyan
}

\newtheorem{theorem}{Theorem}[section]

\newtheorem{lemma}[theorem]{Lemma}
\newtheorem{proposition}[theorem]{Proposition}

\theoremstyle{definition}

\theoremstyle{remark}
\newtheorem*{remark}{Remark}

\newcommand{\Q}{\mathbb{Q}}

\newcommand{\R}{\mathbb{R}}
\newcommand{\C}{\mathbb{C}}
\newcommand{\Z}{\mathbb{Z}}
\newcommand{\h}{\mathbb{H}}

\newcommand{\N}{\mathbb{N}}

\newcommand{\SL}{\text{SL}}
\newcommand{\tr}{\text{tr}}

\newcommand{\disc}{\text{disc}}

\newcommand{\GL}{\text{GL}}

\title{Certain Siegel Cusp Forms with Level are Determined by their Fundamental Fourier Coefficients}
\author{Sidney Washburn\footnote{This research was partially supported by NSF grant DMS-2039316} }

\date{}

\begin{document}

\maketitle 

\begin{abstract}
    We prove that vector-valued Siegel cusp forms for $\Gamma_0^n(N)$ with certain nebentypus are determined by their fundamental Fourier coefficients with discriminants coprime to the level $N$, assuming $N$ is odd and square-free. In the case of genus $3$, we strengthen this to Fourier coefficients corresponding to maximal orders in quaternion algebras. We also prove that Jacobi forms of fundamental index with discriminant coprime to the odd level $N$ are determined by their primitive theta components. 
\end{abstract}

\section{Introduction}

\subsection{Overview}

Nonvanishing results for Fourier coefficients of Siegel modular forms have seen many interesting applications in recent years. One such application is with regards to $L$-functions attached to Siegel cusp forms. Pollack proved in \cite{Pollack} that, assuming the existence of a nonzero Fourier coefficient corresponding to a maximal order in a quaternion algebra, the Spin $L$-function  of a genus $3$, level $1$ Siegel cusp form has a functional equation and meromorphic continuation. The nonvanishing of such a Fourier coefficient in level $1$ was established by B\"ocherer and Das in \cite{BD}, rendering Pollack's result unconditional. \\

This paper makes progress towards the analogous question in level greater than $1$. Precisely, what conditions on the level and the cusp form do we need to guarantee the existence of a nonzero Fourier coefficient corresponding to a maximal order in a quaternion algebra? As an application of this statement in higher level, Eischen, Rosso, and Shah in \cite{eischen} proved the algebraicity of special values of the Spin $L$-function for genus $3$ Siegel cusp forms of arbitrary level. Their work is conditional on the aforementioned nonvanishing condition. To what extent can we make their algebraicity result unconditional? \\

Our main result in this paper (see Theorem \ref{3.1} for the precise statement) is that with assumptions on the nebentypus and level, Siegel cusp forms are determined by their \textit{fundamental} Fourier coefficients with discriminants corpime to the level. Fundamental Fourier coefficients do not necessarily correspond to maximal orders, but they are still of great importance. As an example, the nonvanishing of fundamental Fourier coefficients has implications for the existence of Bessel models (see Remark 1.1 in \cite{Saha}). B\"ocherer and Das (\cite{BD}) proved that Siegel cusp forms are determined by their fundamental Fourier coefficients in level $1$. Anamby (\cite{anamby}) extended their result to arbitrary level when the genus is $2$. This paper is a first step towards answering the analogous question for higher level in arbitrary genus.  \\

\subsection{Proof Strategy and Difficulties}

We give an overview of the proof. The strategy is the same as in \cite{BD}. However, we ran into a couple difficulties employing their method, which we explain in the paragraphs proceeding. The argument is an induction on the genus $n$. Let $F$ be a nonzero genus $n \geq 2$ Siegel cusp form (satisfying the conditions of Theorem \ref{3.1}). The overview is as follows. Via the Taylor development, we construct a genus $n-1$ Siegel cusp form, $G$, whose Fourier coefficients are essentially the Fourier--Jacobi coefficients of $F$. We apply the inductive hypothesis to $G$ so as to acquire a nonzero Fourier--Jacobi coefficient, $\varphi$, with desirable properties. With the help of Theorem \ref{3.3}, we construct a nonzero "twisted" Eichler--Zagier type map, called $h_\epsilon$ (See Propositions \ref{twisted EZ} and \ref{nonvanishing of twist}). $h_\epsilon$ is an elliptic cusp form. From here, we are able to transfer nonvanishing results from the elliptic setting (for $h_\epsilon$) to nonvanishing results of the original Siegel setting (for $F$) and complete the proof. \\

The difficulties arising for us in the situation of higher level are twofold. First, the Fourier--Jacobi coefficients of fundamental index are \textit{not} enough to establish a correspondence between the fundamental Fourier coefficients of a Siegel cusp form and the odd square-free Fourier coefficients of an elliptic cusp form. In our case, we need a Fourier--Jacobi coefficient which not only has fundamental index, but whose index \textit{also} has discriminant coprime to the level. This extra condition is trivial in the level $1$ setting worked out in \cite{BD} and led us to considering twists of the Eichler--Zagier map (see \cite{das_anamby}, \cite{Saha}, \cite{EZ}, \cite{Skoruppa}), whose modularity properties are better suited for working with the level $N$. See Propositions \ref{twisted EZ} and \ref{nonvanishing of twist}. \\

The second difficulty arises with regards to the question of nonvanishing for some twist of the Eichler--Zagier map. Indeed, at least half of the twists $h_\epsilon$ will vanish simply due to parity and weight considerations. To find nonvanishing twists, one must first restrict attention to the class of twists where only the "primitive" components contribute (See Section \ref{The Eichler--Zagier Map}). In this class of twists, we are able to use linear algebra to argue nonvanishing. We note here that this approach \textit{only} works in this subclass of twists. \\

As a last remark, we discuss the problem which originally inspired this work: the nonvanishing of a Fourier coefficient corresponding to a maximal order in a quaternion algebra. Let $F = \sum_T a_F(T) q^T$ be the Fourier expansion of a genus $3$ Siegel cusp form $F$. A Fourier coefficient $a_F(T)$ corresponds to a maximal order in a quaternion algebra when the discriminant of $T$ is an odd prime (\cite{BD}). We succeeded in proving this (See Theorem \ref{3.2}). However, our result does \textit{not} directly apply to the algebraicity results proved in \cite{eischen}. Indeed, the authors there work with the congruence subgroup $\Gamma^0(N)$, while we work with $\Gamma_0(N)$. There is likely an analogous result for $\Gamma^0(N)$, but we have not pursued this yet. \\

\subsection{Further Questions}

A natural question that arises here is the necessity of the assumption on the nebentypus in Theorem \ref{3.1}. We suspect that any assumption on the nebentypus is unnecessary and that it should suffice for the Siegel cusp form to simply be a newform. Indeed, in genus $2$, this is true (See \cite{anamby}). The obstacle one faces in higher genus lies in the induction argument: the construction of genus $n-1$ cusp forms from genus $n$ cusp forms via the Taylor development. Even if one starts with a newform, it is not clear if the resulting genus $n-1$ form is also a newform or has a nonzero newform component. In genus $2$, there are more tools available (unavailable genus $> 2$) that allows one to avoid the inductive argument. Our assumption on the nebentypus forces everything to be new, allowing us to sidestep this issue. \\

As a final remark, we discuss the following question: Given a Siegel cusp form $F$, for which fundamental $T \in \Lambda_n^+$ is the corresponding Fourier coefficient nonzero? First off, both Theorems \ref{3.1} and \ref{3.2} can made quantitative, though we did not write this down. Secondly, while our results say nothing about which $T$ admit nonvanishing Fourier coefficients, one does know that if $F$ is ordinary at $p$ and $a_F(T) \neq 0$ for some given $T$, then $a_G(T) \neq 0$ for any ordinary Siegel cusp form $G$ which lives in a sufficiently small $p$-adic neighborhood of $F$. See \cite{andreatta} for more of this $p$-adic story.

\subsection{Structure of the Paper}

The outline of the paper is as follows. In Section \ref{Notation}, we explain and write down the basic notation and prerequisites needed to understand the proofs and results. In Section \ref{Main results}, we state the main theorems to be proven. In Section \ref{Theorem 3.3 proof}, we prove Theorem \ref{3.3}, whose statement is essential for the nonvanishing of a twist of the Eichler--Zagier map. The proof combines calculations of B\"ocherer and Das \cite{BD} and Anamby \cite{anamby}. Several propositions and lemmas are then collected in Section \ref{props and lemmas}. These intermediary results, form the bridge between the Fourier--Jacobi coefficients and the elliptic setting in the proof of Theorem \ref{3.1}. It is here that we prove the injectivity of a twist of the Eichler--Zagier map. Finally, we conclude with the proofs of Theorems \ref{3.1} and \ref{3.2} in Section \ref{proof of thm 3.1}. 

\subsection{Acknowledgments}

The author thanks his advisor, Ellen Eischen. Frequent conversations, clarifications, and suggestions from her were essential for the continued progress and eventual completion of this paper. The author would also like to thank Soumya Das and Pramath Anamby for helpful questions and clarifications. 

\section{Notation and Preliminaries} \label{Notation}

\subsection{Basic Notation}

Throughout, we use $\exp(\bullet) := e^{2\pi i \cdot \tr(\bullet)}$. The transpose of a vector or matrix $x$ is denoted by $^tx$. For a ring $R$, we will write $R^{(m,n)}$ to denote the set of $m \times n$ matrices whose entries are elements of $R$. For $n \times 1$ column vectors, we will simply write $R^n := R^{(n,1)}$. We will always take the principal square root defined on arguments in $(-\pi,\pi)$, so that $\sqrt{\bullet}$ only returns complex numbers with postive real part and is holomorphic away from the negative real axis. \\

We write $\Lambda_n$ (resp. $\Lambda_n^+)$ to denote the set of $n \times n$ half-integral, symmetric, positive semi-definite (resp. positive definite) matrices. A matrix is half-integral if all of the entries are half-integers, except for the diagonal, whose entries are integers. Equivalently, a symmetric matrix $T$ is half-integral if the quadratic form corresponding to $2T$ represents only even numbers. For $T \in \Lambda_n$ and $x \in \Z^n$, we write $T[x] := \ ^txTx$. \\

The discriminant of $T \in \Lambda_n$ is the quantity 

\[
\disc(T) = d_T := \begin{cases}
    \det(2T) & n \ \text{is even} \\
    \frac{1}{2} \det(2T) & n \ \text{is odd}
\end{cases}
\]

$T$ is called fundamental if $d_T$ is odd and square-free. The \textit{level} of $T$, denoted $\text{lvl}(T)$, is the minimal integer $\ell$ so that $\ell \cdot (2T)^{-1}$ is even. If $T$ is fundamental, then $\mathrm{lvl}(T)$ is explicitly given in terms of the discriminant (See Lemma 2.2 in \cite{BD}): 

\begin{equation} \label{level of T}
    \mathrm{lvl}(T) = \begin{cases}
        d_T & n \ \mathrm{even} \\
        4d_T & n \ \mathrm{odd}
    \end{cases}
\end{equation}

For a Dirichlet character $\chi$ modulo $N$, we will write $\mathfrak{f}_\chi$ to denote the conductor of $\chi$. \\

Frequently, we will use a block decomposition of an $n \times n$ matrix into an upper left $1 \times 1$ block and a lower right $(n-1) \times (n-1)$ block. Unless stated otherwise, we will use fraktur letters to denote the lower right $(n-1) \times (n-1)$ block. e.g.

\[
Z = \begin{pmatrix}
    \ast & \ast \\ \ast & \mathfrak{Z}
\end{pmatrix}, \ T = \begin{pmatrix}
    * & * \\ * & \mathfrak{T}
\end{pmatrix}
\]

\subsection{Siegel Modular Forms}

Fix an $m$-dimensional complex vector space $V$ and let $\rho:\GL_n(\C) \to \GL(V)$ be a polynomial representation on $V$. For such a $\rho$, there exists a maximal integer $k = k(\rho)$ so that $\det^{-k} \otimes \rho$ is polynomial. $k(\rho)$ is called the \textit{determinantal weight} of $\rho$. Let $N > 0$ be an integer, $\chi$ a Dirichlet character mod $N$. The Siegel upper half-space, $\h_n$, is the collection of symmetric complex matrices having positive definite imaginary part:

\[
\h_n := \{Z = X+iY = \ ^tZ \ : \ X,Y \in \R^{(n,n)}, \ Y > 0\}
\]

When $n = 1$, this is the usual upper half-space in $\C$ and we omit the $n$ from the notation (i.e. $\h := \h_1$). \\

A holomorphic function $F:\h_n \to V$ is called a Siegel modular form of genus $n$, automorphy factor $\rho$, with respect to a congruence subgroup $\Gamma \subset \mathrm{Sp_{2n}(\Z)}$ if

\[
F \vert_\rho g = F \ \forall g \in \Gamma.
\]

Here $g = \begin{pmatrix} A & B \\ C & D \end{pmatrix}$ is a block deccomposition of $g$ into $n \times n$ blocks and $F \vert_g$ is the "slash" operator defined as

\[
(F \vert_\rho g)(Z) := \rho(CZ+D)^{-1} F(g \langle Z \rangle)
\]

with $g \langle Z \rangle := (AZ+B)(CZ+D)^{-1}$. When $n = 1$, we require the usual condition of holomorphicity at the cusps. There are three congruence subgroups we frequently consider throughout this paper: 

\[
\Gamma^n(N) := \left\{ \begin{pmatrix} A & B \\ C & D \end{pmatrix} \in \mathrm{Sp_{2n}(\Z)} \ : \ \begin{pmatrix} A & B \\ C & D \end{pmatrix} \equiv \begin{pmatrix} 1 & 0 \\ 0 & 1 \end{pmatrix} \ \mathrm{mod} \ N \right\}
\]

\[
\Gamma_1^n(N) := \left\{ \begin{pmatrix} A & B \\ C & D \end{pmatrix} \in \mathrm{Sp_{2n}(\Z)} \ : \ \begin{pmatrix} A & B \\ C & D \end{pmatrix} \equiv \begin{pmatrix} 1 & \ast \\ 0 & 1 \end{pmatrix} \ \mathrm{mod} \ N \right\}
\]

\[
\Gamma_0^n(N) := \left\{ \begin{pmatrix} A & B \\ C & D \end{pmatrix} \in \mathrm{Sp_{2n}(\Z)} \ : \ \begin{pmatrix} A & B \\ C & D \end{pmatrix} \equiv \begin{pmatrix} \ast & \ast \\ 0 & \ast \end{pmatrix} \ \mathrm{mod} \ N \right\}
\]

When $n = 1$, we omit it from the notation, e.g. we write $\Gamma_0(N) := \Gamma_0^1(N)$. \\

We say a Siegel modular form $F$ for $\Gamma_1^n(N)$ has nebentypus $\chi$ if $F$ satisfies the further condition:

\[
F \vert_g g = \chi(\det(D)) \cdot F \ \ \forall \ g \in \Gamma_0^n(N)
\]

Siegel modular forms have Fourier expansions of the form 

\[
F(Z) = \sum_{T \in \Lambda_n} a_F(T) q^T,
\]

where $q^T := \exp(TZ)$ and $a_F(T) \in \C$. $F$ is called a cusp form when $a_F(T) \neq 0$ implies $T > 0$. We write $S_\rho(\Gamma_0^n(N), \chi)$ to denote the space of Siegel cusp forms with automorphy factor $\rho$ and nebentypus $\chi$. \\

\subsection{Half-integer Weight Modular Forms}

Let $k$ be an odd integer, $N$ a positive integer divisible by $4$. Let $\chi$ be Dirichlet character mod $N$. A modular form of weight $k/2$ with respect to $\Gamma_0(N)$ and nebentypus $\chi$ is a holomorphic function $f:\h \to \C$ satisfying \\

\begin{enumerate}
    \item $f \vert_{k/2} g = \chi(d) \cdot f \ \ \forall \ g \in \Gamma_0(N)$.
    \item $f$ is holomorphic at every cusp. (see \cite{koblitz})
\end{enumerate}

Here, the "slash" operator is defined as follows

\[
f \vert_{k/2}g = j(g,\tau)^{-k} f(g \langle \tau \rangle),
\]

where $j(g,\tau) := \left( \frac{c}{d} \right) \cdot \epsilon_d^{-1} \sqrt{c\tau+d}$. $\left( \frac{c}{d} \right)$ is the (extended) Legendre symbol and $\epsilon_d = 1$ if $d \equiv 1$ mod $4$ and $\epsilon_d = i$ if $d \equiv -1$ mod $4$. \\

$f$ is called \textit{cuspidal} if $f$ vanishes at every cusp. The space of cusp forms of weight $k/2$ with respect to $\Gamma_0(N)$ and nebentypus $\chi$ is denoted $S_{k/2}(\Gamma_0(N), \chi)$. Just as in the integer weight case, $f$ has a Fourier expansion 

\[
f(\tau) = \sum_{n} a_f(n) q^n.
\]

\subsection{Jacobi Forms}

A (scalar-valued) Jacobi form of weight $k$, index $T \in \Lambda_n$, nebentypus $\chi$ mod $N$, for $\Gamma_0(N)$ is a holomorphic function $\phi:\h \times \C^{(1,n)} \to \C$ satisfying the following transformation properties:

\begin{enumerate}
    \item $\phi \vert_{k,T} g = \chi(d) \cdot \phi$ for every $g \in \begin{pmatrix} a & b \\ c & d \end{pmatrix} \in \Gamma_0(N)$
    \item $\phi(\tau,z+x\tau+y) = \exp(-T( \ ^tx\tau x+2 \ ^txz))\phi(\tau,z)$ for every $x,y \in \Z^{(1,n)}$. \label{2}
    \item For every $g = \begin{pmatrix} a & b \\ c & d \end{pmatrix} \in \text{SL}_2(\mathbb{Z})$, there exists a Fourier expansion 
    
    \[
    \phi \vert_{k,T} g = \sum_{n \geq 0} \sum_{r \in \Z^{(1,n)}} c(n,r) q^n\zeta^r
    \]

    and $c(n,r) \neq 0$ implies $\begin{pmatrix}
        n & r/2 \\ ^tr/2 & T
    \end{pmatrix} \geq 0$
\end{enumerate}

Here, $q^n := \exp(n\tau)$, $\zeta^r := \exp(^trz)$ and the slash operator $\phi \vert_{k,T} g$ is defined as follows:

\[(\phi \vert_{k,T} g)(\tau,z) := \phi \left(\frac{a\tau + b}{c\tau+d}, \frac{z}{c\tau+d}\right)(c\tau+d)^{-k} \exp\left(\frac{-Tc\ ^tzz}{c\tau+d}\right).
\]

A Jacobi form $\phi$ as above is cuspidal if, at every cusp, we have that $c(n,r) \neq 0$ implies $\begin{pmatrix}
    n & r/2 \\ ^tr/2 & T
\end{pmatrix} > 0$. We write $J_{k,T}(N,\chi)$ to denote the collection of Jacobi forms (not necessarily cuspidal) of weight $k$, index $T$, nebentypus $\chi$ and level $N$. Property $\ref{2}$ leads to the well-known "theta decomposition" of $\phi$:

\[
\phi(\tau,z) = \sum_{\mu \in \Z^n/2T\Z^n} h_\mu(\tau) \Theta_{\mu,T}(\tau,z)
\]

where $\displaystyle \Theta_{\mu,T}(\tau,z) := \sum_{\ell \in \Z^n} \exp(T(\ ^t(\ell + \tilde{\mu})(\ell + \tilde{\mu}) \tau + 2 \ ^t(\ell + \tilde{\mu})z))$ and $\tilde{\mu} := (2T)^{-1}\mu$. The $h_\mu(\tau)$ are called the \textit{theta components} of $\phi$ are are modular forms of weight $k - n/2$ with respect to a certain congruence subgroup. $h_\mu$ has a Fourier expansion of the form:

\[
h_\mu(\tau) = \sum_{\ell \geq 0} c\left(\ell, \ ^t\mu \right)q^{\ell - T^{-1}[\mu/2]}
\]

The modularity properties of the theta components $h_\mu$ are partially derived from the transformation laws of the theta series $\Theta_\mu$, which we record here:

\begin{equation}
    \Theta_\mu \vert_{n/2,T} \begin{pmatrix}
        1 & 1 \\ 0 & 1
    \end{pmatrix} = \exp((2T)^{-1}[\mu])\Theta_\mu \label{Theta T}
\end{equation} 

and

\begin{equation}
        \Theta_\mu \vert_{n/2,T} \begin{pmatrix}
        0 & -1 \\ 1 & 0 
    \end{pmatrix} = \det(2T)^{-1/2} i^{-n/2} \sum_{\nu \in \Z^n/2T\Z^n} \exp(-^t\nu(2T)^{-1}\mu) \Theta_\nu \label{Theta S}
\end{equation} 

(\ref{Theta T}) is immediate from the definition of $\Theta_{\mu,T}$. (\ref{Theta S}) is a consequence of the Poisson summation formula and can be found in various references, e.g. \cite{ziegler}.

\subsection{The Fourier--Jacobi Expansion}

Siegel modular forms and Jacobi forms are naturally connected via the \textit{Fourier -- Jacobi expansion}. Assume $n > 1$ and let $F \in S_\rho(\Gamma_0^n(N),\chi)$ be a Siegel cusp form with Fourier coefficients $a_F(T)$. Break the input variable $Z \in \h_n$ and the indexing half-integral matrix $T$ into $1\times 1$ and $(n-1) \times (n-1)$ block matrices as follows:

\[
Z = \begin{pmatrix}
    \tau & z \\ ^tz & \mathfrak{Z}
\end{pmatrix}, T = \begin{pmatrix}
    t & r/2 \\ ^tr/2 & \mathfrak{T}
\end{pmatrix}
\]

Here $\tau \in \h$, $\mathfrak{Z} \in \h_{n-1}$, $z \in \C^{(1,n-1)}, t \in \Z_{>0}, r\in \mathbb{Z}^{(1,n-1)},\mathfrak{T}\in \Lambda_{n-1}$. The Fourier--Jacobi expansion of $F$ is the Fourier expansion of $F$ in the variable $\mathfrak{Z}$:

\[
F(Z) = F(\tau,z,\mathfrak{Z}) = \sum_{\mathfrak{T} \in \Lambda_{n-1}} \varphi_{\mathfrak{T}}(\tau,z) q^{\mathfrak{T}}
\]

The Fourier coefficients, now functions $\h \times \C^{(1,n)} \to V$, are called the \textit{Fourier--Jacobi coefficients} of $F$. Crucial for us is the Fourier expansion of $\varphi_\mathfrak{T}$:

\[
\varphi_\mathfrak{T}(\tau,z) = \sum_{n \geq 0} \sum_{r \in \Z^{(1,n)}} a_F\begin{pmatrix}
    n & r/2 \\ ^tr/2 & \mathfrak{T}
\end{pmatrix} q^n \zeta^t
\]

When $\varphi_\mathfrak{T}$ is scalar-valued, the theta components $h_\mu$ of $\varphi_\mathfrak{T}$ have the following Fourier expansions

\[
h_\mu(\tau) = \sum_{n \geq 0} a_F\begin{pmatrix}
    n & ^t\mu/2 \\ \mu/2 & \mathfrak{T}
\end{pmatrix} q^{n - T^{-1} [\mu/2]}
\]

When $F$ is scalar-valued, the $\varphi_\mathfrak{T}$ are also scalar valued. In the vector-valued case, the following proposition will be repeatedly used throughout the paper to reduce ourselves to a scalar-valued Jacobi form: \\

\begin{proposition} (Proposition 3.4 in \cite{BD}) \label{reduce to scalar}
    Let $F \in S_\rho(\Gamma_0^n(N),\chi)$ be a nonzero Siegel modular form. If $\varphi_\mathfrak{T}$ is a nonzero Fourier--Jacobi coefficient of $F$, then there exists a nonzero scalar-valued component of $\varphi_\mathfrak{T}$ which is a scalar-valued Jacobi form of level $\Gamma_0(N)$, index $\mathfrak{T}$, nebentypus $\chi$, and weight $k' \geq k(\rho)$.  
\end{proposition}

\begin{proof}
    B\"ocherer and Das prove this in \cite{BD} in the case of level $1$. The proof for the case of general level is identical and we refer the reader to \cite{BD} for the proof. 
\end{proof}

\section{Statement of Main Results} \label{Main results}

This is our main result: \\

\begin{theorem}\label{3.1}
    Let $F \in S_\rho(\Gamma_0^n(N), \chi)$ be a nonzero Siegel cusp form of odd and square-free level $N$ and arbitrary genus $n$. Assume $k(\rho) - n/2 \geq 2$ Further assume the following conditions on $\chi$:

\begin{enumerate}

\item $\mathfrak{f}_\chi = N$. 

\item For each $p \vert N$, $p$ divides the conductor of $\chi \epsilon_p$, where $\epsilon_p$ is the natural quadratic character attached to the number field $\Q(\sqrt{p})$. 

\end{enumerate}

Then there exists a fundamental $T \in \Lambda_n^+$ so that $a_F(T) \neq 0$ and $(d_T,N) = 1$.

\end{theorem}

In the case of genus $3$, there is a correspondence between $T \in \Lambda_3^+$ and orders in quaternion algebras (see \cite{BD}). Theorem \ref{3.1} and work of B\"ocherer and Das imply that we can specialize to Fourier coefficients corresponding to maximal orders: \\

\begin{theorem}\label{3.2}
    Let $F \in S_\rho(\Gamma_0^3(N), \chi)$ be a nonzero cusp form of odd and square-free level $N$ and genus $3$. Assume that $k(\rho) - 3/2\geq 2$. Further assume that $\chi$ satisfies conditions $(1)$ and $(2)$ as in Theorem $1$. Then there exists a $T  \in \Lambda_3^+$ so that $a_F(T) \neq 0$, $(\text{disc}(T),N) = 1$, and $T$ corresponds to a maximal order in a quaternion algebra. 
\end{theorem} 

Our next result regards the theta components of Jacobi forms and is essential in our analysis of the Eichler--Zagier map: \\

\begin{theorem}\label{3.3}
    Let $\varphi \in J^{\text{cusp}}_{k, T}(\Gamma_0(N),\chi)$ be a nonzero Jacobi form with fundamental index $T \in \Lambda_n^+$, nebentypus $\chi$, and odd level $N$. If $(\disc(T),N) = 1$, then there exists $\mu \in \Z^n/2T\Z^n$ with so that $h_\mu \neq 0$ and $T^{-1}[\mu/2]$ has maximum possible denominator, equal to $\mathrm{lvl}(T)$.
\end{theorem} 

B\"ocherer and Das proved this in \cite{BD} in the case of level $1$, Anamby proved this in \cite{anamby} in the case of scalar index, arbitrary level. 

\section{Proof of Theorem \ref{3.3}} \label{Theorem 3.3 proof}
    
For convenience, we recall the statement of Theorem \ref{3.3} here: \\

{\bf Theorem 3.3}: Let $\varphi:\mathbb{H} \times \C^n \to \mathbb{C}$ be a non-zero scalar-valued Jacobi form of odd and square-free level $\Gamma_0(N)$, nebentypus $\chi$, fundamental index $T$. Also assume $(N,2d_T) = 1$. Let $\varphi = \sum_\mu h_\mu \Theta_\mu$ be the theta decomposition, $\mu \in \mathbb{Z}^{n}/2T\mathbb{Z}^n$. Then there exists $\mu$ so that $h_\mu \neq 0$ and $T^{-1}[\mu/2]$ has maximal possible denominator, equal to $\mathrm{lvl}(T)$ \\

\begin{proof}

We proceed by reducing to a calculation already carried out by Anamby in \cite{anamby}. This reduction follows the same steps of B\"ocherer and Das \cite{BD}.\\

From now on, if $\mu$ has the property that $T^{-1}[\mu/2]$ has maximal possible denominator, we call $\mu$ \textit{primitive} and the corresponding theta component $h_\mu$ a \textit{primitive component of $\varphi$}. We proceed in aim of a contradiction. \\

Assume all the primitive components of $\varphi$ vanish. From the transformation formulas (\ref{Theta T}) and (\ref{Theta S}) of the $\Theta_\mu$ and the modularity of $\varphi_T$ under $\begin{pmatrix}
    1 & 0 \\ N & 1
\end{pmatrix}$, we find that 

\[
(N\tau + 1)^k \sum_\eta h_\eta \Theta_\eta = \sum_\mu h_\mu\left(\frac{\tau}{N\tau+1}\right) \left(\det(2T)^{-1}(N\tau+1)^{n/2} \sum_{\eta,\nu} \exp\left(^t (\mu - \eta - N\frac{\nu}{2})(2T)^{-1}\nu\right) \Theta_\eta\right)
\]

Comparing coefficients, we acquire 

\[
(N\tau+1)^kh_\eta(\tau) = \det(2T)^{-1} (N\tau+1)^{(n-1)/2} \sum_{\mu, \nu} h_\mu\left(\frac{\tau}{N\tau+1}\right) \exp\left(^t( \mu - \eta - N\frac{\nu}{2})(2T)^{-1}\nu\right)
\]

Now sending $\tau \mapsto \tau/(-N\tau +1)$, this reduces to 

\[
h_\eta\left(\frac{\tau}{-N\tau+1}\right) = \det(2T)^{-1} (-N\tau+1)^{n/2 - k} \sum_\mu \varepsilon(\mu,\eta) h_\mu(\tau)
\]

where $\displaystyle \varepsilon(\mu,\eta) := \sum_\nu \exp\left(^t (\mu - \eta - N\frac{\nu}{2})(2T)^{-1}\nu\right)$. Assuming that $h_\eta = 0$ for every $\eta$ primitive, we acquire the relation 

\begin{equation}
    0 = \sum_\mu \varepsilon(\mu,\eta) h_\mu(\tau) \label{nonreduced relation}
\end{equation}

for every $\mu \in \mathbb{Z}^n/2T\mathbb{Z}^n$ and every primitive $\eta$. For every integer $t$, we see directly from the Fourier expansion of $h_\mu$ that $h_\mu(\tau + t) = \exp(- T^{-1}[\mu/2] t) h_\mu(\tau)$ for every integer $t$. Note that $t \mapsto \exp(-T^{-1}[\mu/2] t)$ is a character of $\mathbb{Z}$. With this in mind, we get that for every $t \in \Z$,

\[
    0 = \sum_\mu \varepsilon(\mu, \eta) \exp(-T^{-1}[\mu/2] t) h_\mu(\tau) 
\]

By linear independence of characters, we conclude that for every fixed $\mu_0$, we have 

\begin{equation}
0 = \sum_{\mu \sim \mu_0} \varepsilon(\mu,\eta) h_\mu(\tau) \label{reduced relation}
\end{equation}

where $\mu \sim \mu_0$ if and only if $t \mapsto \exp(-T^{-1}[\mu/2] t)$ and $t \mapsto \exp(-T^{-1}[\mu_0/2] t)$ define the same character of $\mathbb{Z}$. This is equivalent to the condition $T^{-1}[\mu/2] - T^{-1}[\mu_0/2] \in \mathbb{Z}$. \\

We will have reached a contradiction if the matrix $(\varepsilon(\mu,\eta))_{\eta,\mu}$ has maximal rank for every $\mu_0$ imprimitive, where $\eta$ runs over primitive columns and $\mu$ runs over columns equivalent to $\mu_0$. Indeed, this would imply that every imprimitive theta component vanishes and hence \textit{all} the theta components vanish. \\

We now split into the cases when $n$ is even and $n$ is odd. \\

\underline{$n$ odd:} In this case, $\mathrm{lvl}(T) = 4d_T$ (see Equation \ref{level of T}). We first simplify $\varepsilon(\mu,\eta)$ by reducing to the case of scalars (i.e. $\mu,\nu,\eta,T$ are all scalars). Put $M := 2T$, $U \in \text{SL}_n(\mathbb{Z})$, $\widetilde{M} := M[U]$. Pick $U$ so that, for some $f \geq 2$, $\widetilde{M} \equiv \text{diag}( \ast, \ast, \cdots , \ast, \ast p)$ mod $p^f$ for each odd prime divisor $p \vert \det(\widetilde{M}) = 2d_T$ and $\widetilde{M} \equiv \mathbb{H}\perp \mathbb{H}\perp \cdots \perp (\mathbb{H}\ \text{or} \ \mathbb{F}) \perp \ast 2$ mod $2^f$. Here $\ast$ denotes a $p$-adic unit, $\h := 2xy$, and $\mathbb{F} = 2x^2 + 2xy + 2y^2$. Refer to the proof of Proposition 3.5 in \cite{BD} for details on why such a $U$ exists.\\

Changing variables $\widetilde{( \cdot )} \ = \ ^tU ( \cdot )$, we get 

\[
\varepsilon(\mu,\nu) = \varepsilon(\tilde{\mu},\tilde{\nu}) := \sum_{\tilde{\nu}} \exp\left(^t (\tilde{\mu} - \tilde{\eta} - N\frac{\tilde{\nu}}{2})\widetilde{M}^{-1} \tilde{\nu}\right).
\]

By the congruence conditions on $\widetilde{M}$, the collection $\{\tilde{\nu}_t \ \vert \ t \ \text{mod} \ 2d_T\}$, $\tilde{\nu}_t := \ ^t (0 \ 0 \ \cdots \ 0 \ t)$, forms a set of representatives for $\mathbb{Z}^{n} / \widetilde{M}\mathbb{Z}^{n}$.\\

Now $\varepsilon(\tilde{\mu},\tilde{\nu})$ evaluates to 

\[
\varepsilon(\tilde{\mu}_s,\tilde{\eta}_r) = \sum_{t \ \text{mod} \ 2d_T} \exp\left( ^t(\tilde{\mu}_s - \tilde{\eta}_r - N\frac{\tilde{\nu}_t}{2}) \widetilde{M}^{-1} \tilde{\nu}_t\right)
\]

Let $m$ be the $n \times n$ entry of the adjugate of $\widetilde{M}$. Then by Cramer's rule, we have

\[
\tilde{\mu}_s \widetilde{M}^{-1} \tilde{\nu}_t = smt / (2d_T)
\]

So 

\[
\varepsilon(\tilde{\mu}_s,\tilde{\eta}_r) = \varepsilon(s,r) = \sum_t e^{\frac{2\pi i }{2 d_T} m((s-r)t  - Nt^2/2)}
\]

\[
= \frac{1}{2} \sum_{t \ \text{mod} \ 4 d_T} e^{\frac{2\pi i }{4 d_T} m(2(s-r)t  - Nt^2)} = \frac{1}{2} G(-Nm, 2(s-r)m, 4d_T)
\]

Here, $G(a,b,c) = \sum_{t \ \text{mod} \ c} e^{\frac{2\pi i }{c} (at^2 + bt)}$ is the generalized quadratic Gauss sum. In this setting of scalars, $r$ is primitive iff $(r,2d_T) = 1$ and $s \sim s_0$ iff $s^2 - s_0^2 \equiv 0 \ \text{mod} \ 4d_T$ (both facts arise from fact that $m$ is coprime to $4d_T$). Now the problem has turned into showing that for every fixed imprimitive $s_0$, the matrix $(\varepsilon(s,r))_{r,s}$ has maximal rank, where $r$ runs over $(\mathbb{Z}/2d_T)^\times$ and $s$ runs over $\mathbb{Z}/2d_T$ with $s^2 = s_0^2$ mod $4d_T$. Note that the number of rows is $\prod_{p \vert 2d_T} (p-1)$ and the number of columns is $2^{t'}$, where $t'$ is the number of odd primes dividing $2d_T$ but not $s_0$. Thus, the number of rows is at most the number of columns and so proving maximality of the rank really does suffice for the desired contradiction. \\

We do one more simplification of $\varepsilon(s,r)$. By completing the square along with a calculation from Anamby, we have 

\[
\varepsilon(s,r) = G(-Nm,0,4d_T) \cdot \exp \left({-m \overline{N} \frac{(s-r)^2}{4d_T}}\right)
\]

where $\overline{N} \in \Z$ so that $\overline{N} N \equiv 1$ mod $4d_T$. From this, we see that it suffices to show the matrix 

\[
\left(\exp \left({-m \overline{M} \frac{(s-r)^2}{m_1^2}}/(4m_2)\right)\right)_{r,s}
\]

has maximal rank. This is done in \cite{anamby} (Lemma 3.3, take $\ell = -m \overline{M})$). Thus, we have reached a contradiction since $s_0$ may be chosen to be any imprimitive. \\

\underline{$n$ even}: In this case, $\mathrm{lvl}(T) = d_T$. We proceed as in the odd case and pick $U \in \SL_n(\Z)$ so that $\widetilde{M} \equiv \text{diag}(\ast,...,\ast, \ast p)$ mod $p^f$ for each odd prime divisor $p \vert \det(\widetilde{M}) = d_T$ and $\widetilde{M} \equiv \h \perp \cdots \perp (\h \ \text{or} \ \mathbb{F})$ mod $2^f$. As before, the set $\{\tilde{\nu}_t \ \vert \ t \in \Z/d_T\}$ forms a set of representatives for $\Z^n/\widetilde{M}\Z^n$. Changing variables $\widetilde{( \cdot )} = \ ^tU ( \cdot )$, we find 

\[
\varepsilon(\tilde{\mu}_s, \tilde{\eta}_r) = \varepsilon(s,r) = G(-Nm', 0, d_T) \exp(-Nm'(s-r)^2/d_T).
\]

Here, $m' := m/2$, where $m$ is the $n \times n$ entry of the adjugate of $\widetilde{M}$. Note that $m$ is even in this case because of the form of $\widetilde{M}$ mod $2^f$. Hence, it remains to show that for every $s_0 \in \Z/d_T\Z$ not coprime to $d_T$, the matrix 

\[
\left( \exp(-Nm'(s-r)^2/d_T )\right)_{r,s}
\]

has maximal rank, where $r$ runs over $(\Z/d_T\Z)^{\times}$ and $s$ runs over $\Z/d_T\Z$ so that $s^2 \equiv s_0^2$ mod $d_T$. Anamby does not quite prove this in \cite{anamby}, so we prove it here, formulated as the following lemma: \\

\begin{lemma} \label{rank lemma}
    Let $d$ be an odd, square-free integer and $a \in \Z$ coprime to $d$. For an integer $s_0$, define the matrix $\mathcal{E}_{a,d}(s_0) := \left( \exp(a(s-r)^2/d)\right)_{r,s}$, where $r$ runs over $(\Z/dZ)^\times$ and $s$ runs over $\Z/d\Z$ so that $s^2 \equiv s_0^2$ mod $d$. Then $\mathcal{E}_{a,d}(s_0)$ has maximal rank, equal to the number of columns, $2^t$, where $t$ is the number of primes dividing $d$ but not $s_0$.  
\end{lemma}

The proof of Lemma \ref{rank lemma} is an induction on the number of prime divisors of $d$. If $d = 1$, $\mathcal{E}_{a,d}(s_0)$ is a column and so the maximal rank condition is equivalent to $\mathcal{E}_{a,d}(s_0)$ having a nonzero entry. Every entry is nonzero, so this is clear. \\

If $d = p$ is an odd prime, we have two cases: $p \vert s_0$ and $p 
\not\vert s_0$. When $p \vert s_0$, $\mathcal{E}_{a,d}(s_0)$ is again a column and the claim is clear. If $p \not\vert s_0$, then $\mathcal{E}_{a,d}(s_0)$ has two columns, one corresponding to $s_0$ and the other corresponding to $-s_0$. The $2 \times 2$ minors are of the form:

\[
\exp\left( \frac{a((s_0-r)^2 + (s_0+r')^2)}{p} \right) - \exp\left( \frac{a((s_0-r')^2 + (s_0+r)^2)}{p}\right)
\]

where $r,r' \in (\Z/p)^\times$. With a short calculation, noting that $p$ is odd and $(s_0,p) = 1$, we see that the expression above vanishes iff $r = r'$ mod $p$. Hence, $\mathcal{E}_{a,d}(s_0)$ has a nonzero $2 \times 2$ minor and has rank $2$. \\

Now assume that $\mathcal{E}_{a,d}(s_0)$ has maximal rank whenever $d$ is composed of no more than $t \geq 1$ prime factors. Suppose that $d = p_1 p_2 \cdots p_t p_{t+1}$ is the prime decomposition of $d$. Put $d' := d/p_{t+1}$. Decompose $s = p_{t+1}s' + d's''$, where $s'$ is determined modulo $d'$ and $s''$ is determined modulo $p_{t+1}$. Similarly, decompose $r = p_{t+1}r' + d' r''$. Then 

\[
(s-r)^2 = p_{t+1}^2(s'-r')^2 + (d')^2(s''-r'')^2 \ \text{mod} \ d
\]

As $r',r''$ and $s',s''$ run over the sets $(\Z/d'\Z)^\times$, $(\Z/p_{t+1}\Z)^\times$ and $\{s' \in \Z/d'\Z \ \vert \ (s')^2 \equiv (s_0')^2 \ \text{mod} \ d'\}$, $\{s'' \in \Z/p_{t+1}\Z \ \vert \ (s'')^2 \equiv (s_0'')^2 \ \text{mod} \ p_{t+1}\}$, $r$ and $s$ run over the sets $(\Z/d\Z)^\times$ and $\{s \in \Z/d\Z \ \vert \ s^2 \equiv s_0^2 \ \text{mod} \ d\}$. Hence, $\mathcal{E}_{a,d}(s_0)$ decomposes into the tensor product 

\[
\mathcal{E}_{a,d}(s_0) = \mathcal{E}_{ap_{t+1}, d'}(s_0') \otimes \mathcal{E}_{ad', p_{t+1}}(s_0'')
\]

By the well-known formula $\text{rank}(A \otimes B) = \text{rank}A \cdot \text{rank}B$, Lemma \ref{rank lemma} is proven and this completes the proof of Theorem \ref{3.3}

\end{proof}

\section{Some Propositions and Lemmas} \label{props and lemmas}

In this section, we collect some results essential to Theorems \ref{3.1} and \ref{3.2}. 

\subsection{Elliptic Modular Forms}

Two more lemmas we need are due to Serre and Stark \cite{SS} in the half-integer setting and are classical in the integer setting. \\

\begin{lemma}[Lemma 1 in \cite{Yamana}] \label{4.4}
    Let $k \in \frac{1}{2}\Z$ and $f \in S_{k}(\Gamma_0(N),\chi)$. For any prime $p$, the function $\displaystyle g(\tau) := \sum_{(n,p) = 1} a_f(n) q^n \in S_{k-1/2}(\Gamma_0(Np^2),\chi)$. 
\end{lemma} 

\begin{lemma}[(Lemma 7 in \cite{SS}] \label{4.5}
    Let $k \in \Z$. $f \in S_{k-1/2}(\Gamma_0(4N),\chi)$. Fix a prime $p$. Suppose that $a_f(n) = 0$ whenever $(n,p) = 1$. Then $p \vert \frac{N}{4}$, $\chi \varepsilon_p$ has conductor dividing $N/p$, and $g(\tau) := f(\tau/p) \in S_{k-1/2}(\Gamma_0(N/p), \chi \varepsilon_p)$. Here, $\varepsilon_p$ is the unique quadratic character attached to $\mathbb{Q}(\sqrt{p})$.\\
\end{lemma}

We also need the following easy modification of results from Saha \cite{Saha} and B\"ocherer and Das \cite{BD}. \\

\begin{lemma} \label{4.2}
    Let $k \in \frac{1}{2}\Z$ and $f \in S_{k}(\Gamma_0(4N),\chi)$ be an elliptic cusp form. Assume that $k \geq 5/2$. Suppose that $a_f(n) \neq 0$ for some $(n,N) = 1$. Then there exists $m$ odd and square-free so that $(m,N) = 1$ and $a_f(m) \neq 0$. 
\end{lemma}

\begin{proof}

Let $q_1,...,q_t$ be the distinct primes dividing $N$. Define $g_{-1} := f$, $\displaystyle g_0 := \sum_{(\ell,2) = 1} a_{g_0}(\ell) q^\ell$, and $\displaystyle g_{k+1}(\tau) := \sum_{(\ell, q_{k+1})=1} a_{g_k}(\ell) q^\ell$. We have $g_t \neq 0$ because $a_{g_t}(n) =a_f(n) \neq 0$. By Lemma \ref{4.4}, $g_t$ is modular with respect to $\Gamma_0(16N(q_1 \cdots q_t)^2)$ and nebentypus $\chi$. We claim $g_t$ has nonzero odd and square-free Fourier coefficients, proving the lemma. When $k$ is a half-integer, this follows from Proposition 3.7 of \cite{Saha}. If $k$ is an integer, this follows from Theorem 4.6 in \cite{BD}. 

\end{proof}

\subsection{The Eichler--Zagier Map} \label{The Eichler--Zagier Map}

This section develops a small part of the theory for the Eichler--Zagier map in arbitrary genus and level. The ideas appearing here are inspired by work due to Skoruppa, in \cite{Skoruppa}, and Eichler and Zagier, in \cite{EZ}, who worked out the case of level $1$. \\

We begin with a lemma:

\begin{lemma} \label{primitivity lemma}
    Let $T \in \Lambda_n^+$ be fundamental with discriminant $d_T$. Fix any $\Z$-module isomorphism $\alpha:\Z^n/(2T)\Z^n \to \Z/d'\Z$, where $d' = d_T$ if $n$ is even and $d' = 2d_T$ if $n$ is odd. Then $\mu \in \Z^n/(2T)\Z^n$ is primitive if and only if $\alpha(\mu) \in (\Z/d'\Z)^\times$. 
\end{lemma}

\begin{proof}
    We split into cases when $n$ is even or odd. 

    \underline{$n$ even}: In this case $d_T = \det(2T)$. By standard linear algebra (e.g. Smith normal form), $\alpha$ exists. By Proposition 2.3 of \cite{BD}, there exists a primitive $\mu \in \Z^n/(2T)\Z^n$ and $T^{-1}[\mu/2]$ has denominator $d_T$. Because $d_T$ is square-free, we have that for every $0 < m < d_T$, $T^{-1}[m\mu/2] \notin \Z$. In particular, $m\mu \notin (2T)\Z^n$. So $\mu$ is a generator of $\Z^n/(2T)\Z^n$ and we necessarily have $\alpha(\mu) \in (\Z/d_T\Z)^\times$. The set $\{m\mu \ : \ 0 \leq m < d_T\}$ forms a set of representatives for $\Z^n/(2T)\Z^n$. From the relation $T^{-1}[m\mu/2] = m^2T^{-1}[\mu/2]$, we see that $m\mu$ is primitive if and only if $(m,d_T) = 1$ if and only if $\alpha(m\mu) = m\alpha(\mu) \in (\Z/d_T\Z)^\times$. \\

    \underline{$n$ odd}: In this case, $2d_T = \det(2T)$. As before, $\alpha$ exists and now $\Z^n/(2T)\Z^n \cong \Z/2d_T\Z$. Again, there exists $\mu$ which is primitive. Here, this means now that $T^{-1}[\mu/2]$ has denominator $4d_T$. One follows the exact argument above to complete the proof.

\end{proof}

The next Proposition \ref{twisted EZ} coincides with the work of Ramakrishnan and Manickam in \cite{RM} when the chosen character $\epsilon$ is trivial. However, as noted in \cite{BA}, many of the proofs are omitted and it is known that the generalized Maaß lift used in their paper is incorrect. Hence, we include the proof of the case when $\varepsilon$ is trivial.

\begin{proposition} \label{twisted EZ}
    Let $\phi \in J_{k,T}(\Gamma_0(N),\chi)$ be a nonzero Jacobi form of fundamental index $T$. Assume $(d_T,N) = 1$ and that $N$ is odd and square-free. Consider the theta expansion of $\phi:$

    \[
    \phi(\tau,z) = \sum_{\mu \in \Z^n/2T\Z^n} h_\mu(\tau) \Theta_\mu(\tau,z)
    \]

    Fix a $\Z$-module isomorphism $\alpha:\Z^n/(2T)\Z^n \to \Z/d'\Z$, where $d'$ is as in Lemma \ref{primitivity lemma}. For a Dirichlet character $\epsilon$ modulo $d'$, extend the definition of $\epsilon$ to $\Z^n/(2T)\Z^n$ via $\alpha$, i.e. $\epsilon(\mu):=\epsilon(\alpha(\mu))$. Define the \textit{twisted} Eichler--Zagier map $h_\epsilon$ as follows: 

    \[
    h_\epsilon(\tau) := \sum_{\mu \in \Z^n/(2T)\Z^n} \epsilon(\mu) h_\mu(\mathrm{lvl}(T) \cdot \tau)
    \]

    Then $h_\epsilon \in M_{k-n/2}(\Gamma_0(\mathrm{lvl}(T)^2N), \chi \cdot \omega)$, where $\omega$ is a Dirichlet character modulo $\mathrm{lvl(T)}$, given below:
    
    \[\omega(d) = 
    \begin{cases}
        (\epsilon\cdot\chi_T^{-1})(d)\cdot\left(\dfrac{-1}{d} \right)^{(2k-(n+1))/2} & n \ \mathrm{odd} \\ 
        (\epsilon\cdot\chi_T^{-1})(d) & n \ \mathrm{even}
    \end{cases}
    \]
    
    Here, $\chi_T$ is a certain Dirichlet character modulo $\mathrm{lvl}(T)$ (See \cite{Ajouz}).
\end{proposition}

\begin{proof}
    When $g = \begin{pmatrix} a & b \\ c & d \end{pmatrix}\in \Gamma_0(\text{lvl}(T))$, we have the following transformation formula for $\Theta_\mu$ (Theorem 2.3.4 in \cite{Ajouz}):

\[
\begin{cases}
    \Theta_{d\mu}\vert_{n/2, T}g = \left(\frac{c}{d}\right) \epsilon_d^{-1} \chi_T(d)\exp(bdT^{-1}[\mu/2]) \Theta_\mu(\tau,z) & n \ \text{odd} \\ 
    \Theta_{d\mu}\vert_{n/2, T}g = \chi_T(d)\exp(bdT^{-1}[\mu/2]) \Theta_\mu(\tau,z) & n \ \text{even}
\end{cases}
\]

For $g \in \Gamma_0(N)$, we have 

\[
\phi \vert_{k,T} g = \chi(d) \phi(\tau,z)
\]

Together, we find that for $g \in \Gamma_0(\text{lvl}(T) \cdot N)$, 

\begin{equation}
\begin{cases} \label{theta component law}
    h_{d\mu}(g \langle \tau \rangle) = \chi(d)\chi_T^{-1}(d) \left(\frac{c}{d}\right) \epsilon_d^{-1}(d) (c \tau + d)^{k-n/2} \exp(-bdT^{-1}[\mu/2]) h_\mu(\tau) & n \ \text{odd} \\ 
    h_{d\mu}(g \langle \tau \rangle) = \chi(d)\chi_T^{-1}(d) (c \tau + d)^{k-n/2} \exp(-bdT^{-1}[\mu/2]) h_\mu(\tau) & n \ \text{even}
\end{cases}
\end{equation}

Summing the theta components together, we acquire the function equation below for every $g \in \Gamma_0(\text{lvl}(T)N) \cap \Gamma^0(\text{lvl}(T))$:

\[
\begin{cases}
    \sum_\mu \varepsilon(d\mu)h_{d\mu}(g \langle \tau \rangle) = \varepsilon(d)\chi(d) \chi_T^{-1}(d) \left(\frac{c}{d}\right) \epsilon_d^{-1} (c\tau+d)^{k-n/2} \sum_\mu \varepsilon(\mu)h_\mu(\tau) & n \ \text{odd} \\ 
    \sum_\mu \varepsilon(d\mu)h_{d\mu}(g \langle \tau \rangle) = \varepsilon(d)\chi(d) \chi_T^{-1}(d) (c\tau+d)^{k-n/2} \sum_\mu \varepsilon(\mu)h_\mu(\tau) & n \ \text{even}
\end{cases}
\]

In conclusion, we have shown that $h_\varepsilon(\tau/\text{lvl}(T)) = \sum_\mu \varepsilon(\mu)h_\mu(\tau)$ is modular with respect to $\Gamma_0(\text{lvl}(T)N) \cap \Gamma^0(\text{lvl}(T))$, having weight $k-n/2$ and nebentypus $\chi \cdot \omega$. Here, $\Gamma^0(M)$ is the subgroup of matrices which are lower triangular modulo $M$. Hence, $h_\varepsilon$ is modular with respect to $\Gamma_0(\text{lvl}(T)^2N)$. \\

\end{proof}

In general, we expect that a large proportion of the twists, $h_\epsilon$, will vanish. Our next proposition ensures that at least one of the twists does not vanish. 

\begin{proposition} \label{nonvanishing of twist}
    Let $\phi \in J_{k,T}(\Gamma_0(N),\chi)$ be a nonzero Jaccobi form of fundamental index $T$. Assume that $(d_T,N) = 1$ and that $N$ is odd and square-free. Then there exists a Dirichlet character $\epsilon$ modulo $d'$ so that the corresponding twisted Eichler--Zagier map $h_\epsilon$ does not vanish. Here, $d' = d_T$ or $2d_T$ just as in Lemma \ref{primitivity lemma}.
\end{proposition}

\begin{proof}
    Suppose $h_\epsilon = 0$ for every $\epsilon$. This is equivalent to the vanishing of the following expression:

    \[
    \left( \epsilon(\mu) \right)_{\epsilon,\mu} \cdot (h_\mu)_\mu = A(d') \cdot (h_\mu)_\mu
    \]

    Here, the matrix $A(d')$ has rows indexed by Dirichlet characters $\epsilon$ modulo $d'$ and columns indexed by primitive components $\mu \in \Z^n/(2T)\Z^n$. By Lemma \ref{primitivity lemma}, there are $\phi(d')$ columns, which is equal to the number of rows. So $A(d')$ is a square matrix. From Theorem $\ref{3.3}$, the column $(h_\mu)_\mu$ does not vanish. Hence, we will have reached a contradiction if $A(d')$ is invertible. We prove this statement now. \\

    We induct on the number of prime factors of $d'$. When $d' = 1$ or $2$, $A(d')$ is simply the number $1$ and this is obvious. Suppose that $d' = p$ is an odd prime. Then, up to rearranging rows and columns, $A(p)$ takes the following form: 

    \[
    A(p) = \begin{pmatrix}
        1 & 1 & 1 & \cdots & 1 \\
        1 & \zeta_{p-1} & \zeta_{p-1}^2 & \cdots & \zeta_{p-1}^{p-2} \\ 1 & \zeta_{p-1}^2 & (\zeta_{p-1}^2)^2 & \cdots & (\zeta_{p-1}^2)^{p-2} \\ \vdots & \vdots & \ddots & \cdots & \vdots\\ 1 & \zeta_{p-1}^{p-2} & (\zeta_{p-1}^{p-2})^2 & \cdots & (\zeta_{p-1}^{p-2})^{p-2}
    \end{pmatrix}
    \]

    This follows by noting that $(\Z/p\Z)^\times$ is cyclic of order $p-1$. So $A(p)$ is a Vandermonde matrix with determinant $\displaystyle \prod_{0 \leq i < j \leq p-2} (\zeta_{p-1}^i - \zeta_{p-1}^j)$. Of course, this is nonzero and so $A(p)$ is invertible. \\

    In general, let $d' = p_1 \cdots p_t$ be the prime factorization of $d'$. By the Chinese Remainder Theorem, we have $A(d') = A(d'/p_t) \otimes A(p_t)$. By induction on $t$ along with the well-known rank identity $\text{rank}(A \otimes B) = \text{rank}(A) \cdot \text{rank}(B)$, $A(d')$ is invertible and the proof is complete. 
\end{proof}

\begin{remark} \label{seeing the level}
    From the proof of Proposition \ref{twisted EZ}, Equation (\ref{theta component law}), one sees that the theta components alone are modular for the group $\Gamma(\text{lvl}(T)N)$. Via the Eichler--Zagier map, their modularity is improved to $\Gamma_0(\text{lvl}^2(T)N)$ \textit{with nebentypus involving} $\chi$. The inherited nebentypus allows us to control where the resulting elliptic cusp form $h_\varepsilon$ lands. Importantly, it allows us to ensure that $h_\varepsilon$ is a \textit{newform}. \\

\end{remark}

\begin{remark}
    When defining the Eichler--Zagier maps in Proposition \ref{twisted EZ}, note that we specifically chose the twist $\epsilon$ to be a character modulo $d'$. One can in fact choose characters modulo divisors of $d'$. Say $e$ is such a divisor and $\epsilon$ is a Dirichlet character modulo $e$. Then following a similar argument appearing in \cite{Kramer}, the corresponding Eichler--Zagier twist $h_\epsilon$ in fact has improved level equal to $M\cdot e/d'$, where $M$ is the level appearing in Proposition \ref{twisted EZ}. We note however that despite the improved level, the sum defining $h_\epsilon$ now potentially includes \textit{imprimitive} theta components, whose Fourier expansions are \textit{not} necessarily supported away from $d'$. This becomes important for the proof of Theorems \ref{3.2} and \ref{3.3}. Proposition $\ref{nonvanishing of twist}$ also only applies to Dirichlet characters modulo $d'$. For these reasons, we have omitted the proof of the aforementioned statements for characters of smaller modulus in Proposition \ref{twisted EZ}. 
\end{remark}

\section{Proof of Theorem \ref{3.1}} \label{proof of thm 3.1}

With the preparations of Section \ref{props and lemmas}, we are now ready to prove Theorems \ref{3.1} and \ref{3.2}. We recall the statement of Theorem \ref{3.1} for convenience: \\

{\bf Theorem \ref{3.1}}:
    Let $F \in S_\rho(\Gamma_0^n(N), \chi)$ be a nonzero Siegel cusp form of odd and square-free level $N$ and genus $3$. Assume $k(\rho) - n/2 \geq 2$. Further assume the following conditions on $\chi$:

\begin{enumerate}

\item $\mathfrak{f}_\chi = N$. 

\item For each $p \vert N$, $p$ divides the conductor of $\chi \epsilon_p$, where $\epsilon_p$ is the natural quadratic character attached to the number field $\Q(\sqrt{p})$. 

\end{enumerate}

Then there exists a fundamental $T \in \Lambda_n^+$ so that $a_F(T) \neq 0$ and $(d_T,N) = 1$. \\

\begin{proof}

The argument is essentially the one appearing in \cite{BD}, except we use the twisted Eichler--Zagier map to ensure we can say something with respect to the level. \\

Assume $n = 1$. Up to a change of basis, we may assume that $F$ is valued in $\C^m$ and that the components of $F = (F_1,...,F_m)$ are themselves modular forms of weight $k_1,...,k_m$. So it suffices to prove the theorem statement for scalar-valued cusp forms. Because the nebentypus $\chi$ is primitive, $F$ is a newform. Hence, $F$ has a nonzero Fourier coefficient $a_F(\ell) \neq 0$ with $(\ell, N) = 1$ (see Lemma $2$ in \cite{Yamana}). By Lemma \ref{4.2}, $F$ has a nonzero Fourier coefficient $a_F(\ell')$ with $\ell'$ odd, square-free, and coprime to $N$, completing the base case of genus $n = 1$. \\

Now assume $n \geq 2$. For $Z \in \mathbb{H}_n$, write $Z = \begin{pmatrix} \tau & z \\ ^t z & \mathfrak{Z} \end{pmatrix}$ as before. Taylor expand $F$ around $z = 0$:

\[
F(\tau, z, \mathfrak{Z}) = \sum_{\lambda \in \N^{n-1}} F_\lambda(\tau, \mathfrak{Z}) z^\lambda.
\]

Let $\nu_0$ be the minimal integer so that there exists $\lambda$ with $\nu(\lambda) = \nu_0$ and $F_{\lambda} \neq 0$. Define a vector-valued function $F^{\circ}:\h_1 \times \h_{n-1} \to V \otimes_\C \C[x_2,...,x_n]_{\nu_0}$ as follows

\[
F^{\circ}(\tau,\mathfrak{Z}) := \sum_{\lambda: \nu(\lambda) = \nu_0} F_\lambda(\tau,\mathfrak{Z}) \otimes x_2^{\lambda_2} \cdots x_n^{\lambda_n}
\]

Consider the embedding 

\[
(\cdot)^{\downarrow}:\text{Sp}_{2n-2}(\Z) \hookrightarrow \text{Sp}_{2n}(\Z), g = \begin{pmatrix} a & b \\ c & d \end{pmatrix} \mapsto \begin{pmatrix} 1 & & 0 & \\ & a & & b \\ 0 & & 1 & \\ & c & & d \end{pmatrix} = g^\downarrow.
\]

{\bf Proposition (\cite{BD} Prop 3.1):} We have the following transformation law for $F^\circ$:

\[
(F \vert_\rho g^{\downarrow})^\circ (\tau,z) = \rho \begin{pmatrix} 1 & 0 \\ 0 & cz+d \end{pmatrix}^{-1} \otimes \text{Sym}^{\nu_0}(cz+d)^{-1} F^\circ (\tau, g \langle z \rangle). 
\]

Here, $\text{Sym}^{\nu_0}$ is the symmetric power representation of $\text{GL}_{n-1}(\mathbb{C})$ on $\mathbb{C}[x_2,...,x_n]_{\nu_0}$ given by $(A \cdot f)(x) := f(xA)$. \\

Let $\rho'$ denote the restriction of $\rho$ to the embedding 

\[
\text{GL}_{n-1} \hookrightarrow \text{GL}_{n}, g \mapsto \begin{pmatrix} 1 & 0 \\ 0 & g \end{pmatrix}.
\]

Then, restated, the proposition simply says that 

\[
(F \vert_\rho g^{\downarrow})^\circ (\tau, z) = (F^\circ \vert_{\rho' \otimes \text{Sym}^{\nu_0}} g)(\tau, z)
\]

When $g \in \Gamma_0^{n-1}(N)$, we have by the assumption on $F$ that $F \vert_\rho g^{\downarrow} = \chi(\det(d)) F$. So together with the proposition, we have that for every fixed $\tau$, 

\[
F^\circ \vert_{\rho' \otimes \text{Sym}^{\nu_0}} g = \chi(\det(d)) F^\circ
\]

In particular, $F^\circ \in S_{\rho' \otimes \text{Sym}^{\nu_0}}^{n-1}(\Gamma_0^{2}(N)), \chi)$ is a non-zero, genus $2$ Siegel cusp form with nebentypus $\chi$. Because $k(\rho' \otimes \mathrm{Sym}^{\nu_0}) \geq k(\rho)$ and $k(\rho) - n/2 \geq 3$, we have that $k(\rho' \otimes \mathrm{Sym}^{\nu_0}) - (n-1)/2 \geq 3$. So the assumptions of Theorem ??? hold and by the induction hypothesis, $F^\circ$ has a non-zero fundamental Fourier coefficient $a_{F^\circ}(\mathfrak{T})$ with $(\disc(\mathfrak{T}),N) = 1$. \\

Consider now the Fourier--Jacobi expansion $F(Z) = \sum_{\mathfrak{T} \in \Lambda_2} \varphi_{\mathfrak{T}}(\tau, z) q^\mathfrak{T}$. A simple calculation gives that the $\mathfrak{T}$-th Fourier coefficient of $F^\circ$ is given by 

\[
a_{F^\circ}(\mathfrak{T}) = \sum_\lambda \frac{1}{\lambda!}\left(\frac{\partial^\lambda}{\partial z^\lambda}\varphi_{\mathfrak{T}}\right) \vert_{z = 0}(\tau) \otimes x_2^{\lambda_2} \cdots x_n^{\lambda_n}
\]

So $a_{F^\circ}(\mathfrak{T}) \neq 0$ implies $\varphi_\mathfrak{T} \neq 0$. Hence, there exists a nonzero Fourier--Jacobi coefficient $\varphi_{\mathfrak{T}}$ with fundamental index $\mathfrak{T}$ whose discriminant is coprime to $N$.  \\

By Proposition \ref{reduce to scalar}, $\varphi_\mathfrak{T}$ has a nonzero scalar-valued component which is a Jacobi form satisfying the same modularity conditions, with potentially higher weight $k'$. We abuse notation as usual and refer to this nonzero component as $\varphi_\mathfrak{T}$. Consider a theta component $h_\mu$ of $\varphi_\mathfrak{T}$. The Fourier expansion of $h_\mu$ is given:

\[
h_\mu(\tau) = \sum_{\ell > 0} a_F \begin{pmatrix}
    \ell & ^t\mu/2 \\ \mu/2 & \mathfrak{T} 
\end{pmatrix} \exp((\ell - \mathfrak{T}^{-1}[\mu/2])\tau)
\]

From the formula $\det \begin{pmatrix}
    \ell & ^t\mu/2 \\ ^t\mu/2 & \mathfrak{T} 
\end{pmatrix} = (\ell - \mathfrak{T}^{-1}[\mu])\det(\mathfrak{T})$, we see that an odd and square-free Fourier coefficient of $h_\mu(\mathrm{lvl}(\mathfrak{T}) \cdot \tau)$ is exactly a fundamental Fourier coefficient of $F$. Hence, we must prove that $h_\mu$ has a nonzero odd and square-free Fourier coefficient. We prove this now. \\

By Proposition \ref{nonvanishing of twist}, there exists a Dirichlet character $\epsilon$ modulo $\mathrm{lvl}(\mathfrak{T})$ so that $h_\epsilon(\tau) = \sum_\mu \epsilon(\mu) h_\mu(\mathrm{lvl}(\mathfrak{T})\tau)$ is nonzero. $h_\epsilon$ is contained in $S_{k' - (n-1)/2}(\Gamma_0(\mathrm{lvl}(\mathfrak{T})^2N), \chi \cdot \omega)$, where $\omega$ is a Dirichlet character modulo $\mathrm{lvl}(\mathfrak{T})$. Let $N = p_1 \cdots p_t$ be the prime factorization of $N$ and put $g_0 := h_\epsilon$. For $0 < i \leq t$, define $\displaystyle g_{i}(\tau) := \sum_{(n, q_i) = 1} a_{g_{i-1}}(n) q^n$.\\

Because $\epsilon$ is defined modulo $\mathrm{lvl}(\mathfrak{T})$, $h_\epsilon$ is a sum of primitive theta components. The Fourier coefficients of primitive theta components are, by construction, supported away from $\mathrm{lvl}(\mathfrak{T})$. So if $g_t \neq 0$, then the proof is complete by Lemma \ref{4.2}. \\

In the case of odd genus $n$, $h_\epsilon \in S_{k-(n-1)/2} (\Gamma_0(\mathrm{lvl}(\mathfrak{T})^2N, \chi \cdot \omega)$. Note the $h_\epsilon$ has integer weight. Because the Fourier coefficients of $h_\epsilon$ are supported away from $\mathrm{lvl}(\mathfrak{T})$, we have by classical old and newform theory (see e.g. Lemma $1$ and $2$ in \cite{Yamana}) that $h_\epsilon$ has a nonzero Fourier coefficient $a_{h_\epsilon}(\ell) \neq 0$ with $(\ell, N) = 1$. Hence, $g_t$ is nonzero and we're done in this case. \\

In the case of even genus $n$, $h_\epsilon$ has half-integer weight. We proceed by induction on $i$. By construction, $g_0 \neq 0$. Suppose that $g_{i+1} \neq 0$ for some $0 \leq i < t$ and $g_{i} \neq 0$. By repeated application of Lemma \ref{4.4}, $g_i$ is modular with respect to the group $\Gamma_0(\mathrm{lvl}(\mathfrak{T})^2N(q_1 \cdots q_i)^2)$. If $g_{i+1} = 0$, then by Lemma \ref{4.5}, $\chi \varepsilon_{q_{i+1}}$ has conductor dividing $\mathrm{lvl}(\mathfrak{T})^2N(q_1 \cdots q_i)^2/q_{i+1}$. This is impossible however because the latter expression is not divisible by $q_{i+1}$ while the conductor is. So claim is proven.

\end{proof}

In the case of genus $3$, half-integral matrices $T \in \Lambda_3^+$ correspond to orders in quaternion algebras. Corresponding to a maximal order simply means that the discriminant $d_T$ is an odd prime. See \cite{BD} for details. So to prove Theorem \ref{3.2}, one simply needs to refine the statement of Theorem \ref{3.1} to odd prime discriminants. We prove Theorem \ref{3.2} now. In fact, we shall prove a more general statement: \\

{\bf Theorem \ref{3.2} (More general version):} In the setting of Theorem \ref{3.1}, if the genus $n>1$ is odd, then there exists infinitely many odd primes $p$ and infinitely many $T \in \Lambda_n^+$ so that $a_F(T) \neq 0$ and $d_T = p$.

\begin{proof}
    By Theorem \ref{3.1}, there exists a fundamental $\mathfrak{T} \in \Lambda_{n-1}^+$ with $(d_\mathfrak{T},N) = 1$ and the corresponding Fourier--Jacobi coefficient $\varphi_\mathfrak{T} \neq 0$. By Proposition \ref{nonvanishing of twist}, there exists a Dirichlet character $\epsilon$ modulo $d_\mathfrak{T}$ so that $h_\epsilon \neq 0$. $h_\epsilon$ has integer weight (because $n$ is odd), is modular for $\Gamma_0(\mathrm{lvl}(\mathfrak{T})^2N)$ with nebentypus $\chi \cdot \omega$ ($\omega$ a character modulo $\mathrm{lvl}(\mathfrak{T})$), and has Fourier coefficients supported away from $\mathrm{lvl}(\mathfrak{T})$. So by classical old and newform theory, $h_\epsilon$ has a nonzero Fourier coefficient away from $N$. In particular, $h_\epsilon$ is not an oldform. Hence, by Lemma 5.2 in \cite{BD}, $h_\epsilon$ is determined by it's prime Fourier coefficients, completing the proof.
    
\end{proof}

\begin{remark}
    In the case of even genus, Theorem \ref{3.2} is expected (i.e. we expect prime discriminants to determine the Siegel modular form). However, the current proof would require the analogous statement for half-integer weights. This is not known currently. 
\end{remark}

\printbibliography

\end{document}